\documentclass[12pt,a4paper]{amsart}
\usepackage{latexsym}
\usepackage{amssymb}
\usepackage{amscd}
\usepackage[all]{xy}
\usepackage{csquotes}
\usepackage[utf8]{inputenc}
\usepackage[mathscr]{eucal}
\theoremstyle{plain}
\newtheorem{prob}{Problem}[section]
\newtheorem{defin}[prob]{Definition}
\newtheorem{theorem}[prob]{Theorem}

\newtheorem*{theorem*}{Theorem}
\newtheorem*{question*}{Question}
\newtheorem{prop}[prob]{Proposition}
\newtheorem{corollary}[prob]{Corollary}
\newtheorem{lemma}[prob]{Lemma}
\theoremstyle{remark}

\newtheorem{remark}{Remark}

\newcommand{\R}{\mathbb{R}}
\newcommand{\N}{\mathbb{N}}

\newcommand{\adef}{\begin{defin}}
\newcommand{\zdef}{\end{defin}}

\newcommand{\ra}{\rightarrow}

\newcommand{\dist}{\textrm{dist}}

\newcommand{\Dom}{\mathrm{Dom}}
\newcommand{\Ran}{\mathrm{Ran}}

\newcommand{\dPsiker}{\Psi(\ker\Phi)\oplus_{\Omega_{\Psi,\Phi}} X_\Phi}
\newcommand{\dPsiZ}{Z\oplus_{\Omega_{\Psi,\Phi}} X_\Phi}

\title[Differential processes generated by two interpolators]{Differential
processes generated by two interpolators}

\subjclass{Primary: 46B70, 46E30, 46M18.}

\author[J.M.F. Castillo, W.H.G. Corr\^ea, V. Ferenczi, M. Gonz\'alez]{Jes\'us M. F. Castillo,
Willian H. G. Corr\^ea, Valentin Ferenczi, Manuel Gonz\'alez}

\address{Instituto de Matem\'aticas, Universidad de Extremadura,
Avenida de Elvas s/n, 06011 Badajoz, Spain.} \email{castillo@unex.es}

\address{Departamento de Matem\'atica, Instituto de Matem\'atica e
Estat\'\i stica, Universidade de S\~ao Paulo, rua do Mat\~ao 1010,
05508-090 S\~ao Paulo SP, Brazil} \email{willhans@ime.usp.br}

\address{Departamento de Matem\'atica, Instituto de Matem\'atica e
Estat\'\i stica, Universidade de S\~ao Paulo, rua do Mat\~ao 1010,
05508-090 S\~ao Paulo SP, Brazil  \\ and Equipe d'Analyse Fonctionnelle \\
Institut de Math\'ematiques de Jussieu \\
Universit\'e Pierre et Marie Curie - Paris 6 \\
Case 247, 4 place Jussieu \\
75252 Paris Cedex 05 \\
France.} \email{ferenczi@ime.usp.br}

\address{Departamento de Matem\'aticas, Universidad de Cantabria,
Avenida de los Castros s/n, 39071 Santander, Spain.} \email{manuel.gonzalez@unican.es}
\subjclass[2010]{46B70, 46E30, 46M18 }

\thanks{The research of the first author has been supported in part by Project IB16056 de la Junta
de Extremadura; the research of the first and fourth authors has been supported in part by Project
MTM2016-76958, Spain.
The research of the second author has been supported in part by CNPq, grant 140413/2016-2, CAPES,
PDSE program 88881.134107/2016-0, and São Paulo Research Foundation (FAPESP), grants 2016/25574-8
and 2018/03765-1.
The research of the third author has been supported by FAPESP, grants 2013/11390-4, 2015/17216-1,
2016/25574-8 and by CNPq, grant 303034/2015-7}

\begin{document}

\begin{abstract} We study couples of interpolators, the differentials they generate and their associated commutator theorems.
An essential part of our analysis is the study of the intrinsic symmetries of the process. Since we work without any compatibility
or categorical assumption, our results are flexible enough to generalize most known results for commutators or translation operators, in particular those of Cwikel, Kalton, Milman, Rochberg \cite{ckmr} for differential methods and those of Carro, Cerd\`a and Soria \cite{caceso} for compatible interpolators. We also generalize stability and singularity results in \cite{cfg,ccfg,correa} from the complex method to general differential methods and obtain new incomparability results.
\end{abstract}

\maketitle


\section{Introduction}\label{introduction}
In this paper we consider abstract interpolation methods and the differential processes they generate.
Our approach is free of categorical elements in order to avoid unnecessary complications.
We also avoid any compatibility condition. That avoidance is necessary in order to make the interpolator symmetries emerge, and also because considering non-compatible interpolators opens the door to treating translation operators as derivations, or to obtaining the general symmetric commutator results
that can be found in Section \ref{sect:commutator}.

We study couples of interpolators, the differentials they generate and their associated commutator theorems.
An essential part of our analysis is the study of the intrinsic symmetries of the process that allow us to jump from a couple
of interpolators $(\Psi, \Phi)$ to the symmetric couple $(\Phi, \Psi)$. Since we work without any compatibility
or categorical assumption, our results are flexible enough to generalize most known results for commutators or translation operators, in particular those of Cwikel, Kalton, Milman, Rochberg \cite{ckmr} for differential methods and those of Carro, Cerd\`a and Soria \cite{caceso} for compatible interpolators. In the last section we generalize stability and singularity results in \cite{cfg,ccfg,correa} from the complex method to general differential methods and obtain new incomparability results.

\section{Preliminaries}\label{firstresults}

The necessary background on the theory of twisted sums and diagrams can be seen in \cite{castgonz}.
A twisted sum of two quasi-Banach spaces $Y$, $Z$ is  a quasi-Banach space $X$ which has a closed subspace
isomorphic to $Y$ such that the quotient $X/Y$ is isomorphic to $Z$.

An exact sequence (of quasi-Banach spaces and linear continuous operators) is a diagram
$$\begin{CD}
0@>>>  Y@>>> X @>>> Z@>>> 0\end{CD}
$$
in which the kernel of each arrow coincides with the image of the preceding one.
Thus, the open mapping theorem yields that the middle space $X$ is a twisted sum of $Y$ and $Z$. The simplest exact sequence is $0 \to Y \to Y\oplus Z \to Z \to 0$ with embedding $y\to (y,0)$ and
quotient map $(y,z)\to z$. Two exact sequences $0 \to Y \to X_j \to Z \to 0$, $ j = 1, 2$, are said to be
{\it equivalent} if there exists an operator $T:X_1\to X_2$ such that the following diagram commutes:
$$
\begin{CD}
0 @>>>Y@>>>X_1@>>>Z@>>>0\\
&&@| @VVTV @|\\
0 @>>>Y@>>>X_2@>>>Z@>>>0.
\end{CD}$$

By the 3-lemma \cite[p. 3]{castgonz} $T$ must be an isomorphism.
An exact sequence is said to be \emph{trivial}, or to \emph{split}, if it is equivalent to
$0 \to Y \to Y \oplus Z \to Z \to 0$.
Observe that the exact sequence splits if and only if the subspace $Y$ of $X$ is complemented.
\medskip

Let $M, N$ be closed subspaces of a Banach space $Z$, and let $S_M$ denote the unit sphere of $M$.
The \emph{gap $g(M,N)$ between $M$ and $N$} is defined by
$$
g(M,N) = \max \big\{\sup_{x\in S_M}\dist(x,N), \sup_{y\in S_N}\dist(y,M)\big\},
$$
and the \emph{minimum gap $\gamma(M,N)$ between $M$ and $N$} is defined by
$$
\gamma(M,N) = \inf_{u\in M\setminus N} \frac{\dist(u,N)}{\dist(u,M\cap N)}.
$$

Note that $M+N$ is a closed subspace of $Z$ if and only if $\gamma(M,N)>0$ \cite[Theorem IV.4.2]{Kato:80}.

\begin{prop}\label{prop:gap-stab}
Let $M$ and $N$ be closed subspaces of $Z$ such that $M+N$ is closed, and let us denote
$R= (1/2)\min\{\gamma(M,N),\gamma(N,M)\}$.
If $M_1$ and $N_1$ are closed subspaces of $Z$ and $g(M_1,M)+ g(N_1,N)< R$, then
\begin{enumerate}
\item  $M\cap N=\{0\}$ implies $M_1\cap N_1=\{0\}$ and $M_1+N_1$ is closed.
\item  $M+N=Z$ implies $M_1+N_1=Z$.
\end{enumerate}
In particular, if $Z=M\oplus N$ and $g(M_1,M)<R$ then $Z=M_1\oplus N$; i.\ e., the property of
a subspace being complemented is open with respect to the gap.
\end{prop}
\begin{proof}
(1) Since $M\cap N=\{0\}$ and $M+N$ is closed, $\gamma(M,N) =  \inf_{u\in S_M} \dist(u,N)>0$.

Suppose that there exists $u\in M_1\cap N_1$ with $\|u\|=1$.
Since $\dist(u,M) \leq g(M_1,M)$ and $\dist(u,N)\leq g(N_1,N)$, our hypothesis implies
$$
\dist(u,M) + \dist(u,N) < \frac{1}{2} \gamma(M,N)
$$
and $\dist(u,M\cap N)=1$, contradicting \cite[IV Lemma 4.4]{Kato:80}.
Hence $M_1\cap N_1=\{0\}$. A similar argument shows that $M_1+N_1$ is closed.
Indeed, otherwise for every $\varepsilon>0$ we could find $u\in M_1$ and $v\in N_1$
with $\|u\|=\|v\|=1$ and $\|u-v\|<\varepsilon$.
Therefore $\dist(u,M)\leq g(M_1,M)$ and $\dist(u,N)\leq g(N_1,N)+\varepsilon$, and
\cite[IV Lemma 4.4]{Kato:80} would imply
$$
\frac{1}{2}\gamma(M, N) \leq g(M_1, M) +g(N_1, N) +\epsilon
$$
for every $\epsilon > 0$, contradicting the hypothesis.

(2) Let $M^\perp$ denote the annihilator of $M$ in $Z^*$.
Since $M+N=Z$ if and only if $M^\perp\cap N^\perp=\{0\}$ and $M^\perp +N^\perp$ is closed,
$g(M,N)=g(M^\perp,N^\perp)$ and $\gamma(M,N)=\gamma(N^\perp,M^\perp)$ \cite[Chapter IV]{Kato:80},
it is a consequence of (1).
\end{proof}

We refer to \cite[Chapter IV]{Kato:80} for additional information on the gap between subspaces.
\bigskip

Given Banach spaces $X$ and $Y$ and a bounded operator $T:X\to Y$, it is usual in interpolation
theory to endow $T(X)$ with the associated \emph{quotient norm}, defined as follows:
$$
\|Tx\|_q = \|Tx\|_T = \inf \{\|z\| : Tx=Tz\}=\dist(x,\ker T).
$$
Since the operator $T$ induces an isometry $x+\ker T\to Tx$ from $X/\ker T$ onto $(T(X),\|\cdot\|_q)$,
the latter space is a Banach space.

\section{The two symmetric derivations generated by a pair of interpolators}
\label{symmetric-sequences}

We adopt the language of \cite{caceso}, simpler than the one used in \cite{ckmr}, but we omit the functor
terminology because it is not necessary in our context.
We consider an interpolation couple $(X_0, X_1)$ of Banach spaces  continuously embedded into their sum
$\Sigma=X_0+X_1$, which admits a norm $\|\cdot\|_\Sigma$ making it a Banach space (see \cite{BL}). An operator $t:\Sigma\to \Sigma$ is said to \emph{act on the scale} if $t: X_i\to X_i$ is continuous
for $i=0,1$.

An \emph{abstract interpolation method} is generated by
\begin{enumerate}
    \item a Banach space $\mathcal{H}$ of functions
$f:D\to \Sigma$ on a metric space $D$ of parameters such that every operator $t$ acting on the scale generates an operator $T : \mathcal{H} \rightarrow \mathcal{H}$.
    \item an \emph{interpolator on $\mathcal H$},
which is a continuous linear operator $\Phi:\mathcal H\to \Sigma$ such that if $t$ is an operator acting on the scale then $t \circ \Phi = \Phi \circ T$.
\end{enumerate}

We denote by $X_\Phi$ the space $\Phi(\mathcal H)$ endowed with the quotient norm
$$
\|x\|_\Phi =\inf \{\|f\|_\mathcal{H} : f\in\mathcal{H}, \Phi f= x\}\quad \textrm{($x\in X_\Phi$)},
$$
which is a Banach space, and we fix a homogeneous map $B_\Phi: X_\Phi\to \mathcal H$ such that
$\Phi B_\Phi x= x$ and there is $\varepsilon>0$ so that
$\|B_\Phi(x)\|_{\mathcal H} \leq (1+\varepsilon)\|x\|_{\Phi}$ for each $x\in X_\Phi$.

If $t: \Sigma \to \Sigma$ is an operator on the scale then we get an ``interpolated operator" $t: X_\Phi\to X_\Phi$ since
$$
\|t(\Phi f) \|_\Phi = \|t(\Phi B_\Phi \Phi f)\|_\Phi = \|\Phi(TB_\Phi\Phi f)\|_\Phi \leq
\|\Phi\|\cdot\|T\|(1+\varepsilon)\|\Phi(f)\|_\Phi.
$$

To an abstract interpolation method one usually associates a sequence of interpolators.
A typical example is the complex interpolation method for $(X_0,X_1)$, in which $D$ is the unit strip
in the complex plane $\mathbb S = \{z\in \mathbb C: 0< Re(z)< 1\}$, $\mathcal H$ is the Calderon space
of continuous bounded functions $\overline{\mathbb S}\to \Sigma$, which are holomorphic on $\mathbb S$,
and the maps $t\mapsto f(it)\in X_0$ and $t\mapsto f(1+it)\in X_1$ are continuous and bounded, endowed
with the norm $\|f\|_{\mathcal H} =\sup \{ \|f(it)\|_{X_0}, \|f(1+ it)\|_{X_1}\}$, and the interpolator
is $\Phi = \delta_\theta$, the evaluation map at some $\theta$ in the interior of $\mathbb S$.
In this case the sequence of interpolators is formed by the evaluation operators $\delta_\theta^{(n)}$
of the $n^{th}$ derivative at $\theta$.
In this paper we will focus on the ``first" two terms of the sequence of interpolators.
\medskip

Let $(\Psi,\Phi)$ be a pair of interpolators on $\mathcal H$, let
$(\Psi,\Phi): \mathcal{H} \to \Sigma\times \Sigma$ be the map defined by $(\Psi,\Phi)f= (\Psi f,\Phi f)$,
and let $X_{\Psi,\Phi}$ denote the space $(\Psi,\Phi)(\mathcal H)= \{ (\Psi(f), \Phi(f)): f\in \mathcal H\}$,
endowed with the quotient norm.
%
We consider the following commutative diagram:
\begin{equation}\label{psidiagram}
\begin{CD}
&&\ker \Psi \cap \ker \Phi  @= \ker (\Psi, \Phi)\\
&&@VVV @VVV\\
0 @>>> \ker \Phi  @>>> \mathcal H @>{\Phi}>> X_\Phi @>>>0\\
&&@V{\Psi}VV @VV{(\Psi, \Phi)}V @|\\
0 @>>> \Psi(\ker \Phi) @>i>> X_{\Psi,\Phi}  @>p>> X_\Phi @>>>0
\end{CD}\end{equation}
where $\Psi(\ker \Phi)$ is endowed with the quotient norm $\|\cdot\|_q$ associated to the restriction
of $\Psi$ to $\ker\Phi$, and $i,p$ are defined by $i\Psi g=(\Psi g,0)$ and $p(\Psi f,\Phi f)=\Phi f$.
So the rows are exact.

\begin{defin}
The \emph{derivation} associated to the pair $(\Psi, \Phi)$ is the map
$\Omega_{\Psi,\Phi}: X_\Phi\to \Sigma$ given by $\Omega_{\Psi,\Phi} = \Psi B_\Phi$.
\end{defin}

The derivation $\Omega_{\Psi,\Phi}$ generates the so-called \emph{derived space}
$$
d\Omega_{\Psi,\Phi}:=\dPsiker =\{(w,x)\in\Sigma\times X_\Phi: w-\Omega_{\Psi,\Phi}x\in \Psi(\ker\Phi)\},
$$
endowed with $\|(w,x)\|_{\Omega_{\Psi,\Phi}}= \|w-\Omega_{\Psi,\Phi} x\|_q + \|x\|_\Phi$.

\begin{remark}\label{rem:qnorm}
$\|(\cdot,\cdot)\|_{\Omega_{\Psi,\Phi}}$ is a quasi-norm because
$B_\Phi(x+y)-B_\Phi(x)-B_\Phi(y)\in \ker \Phi$, hence
\begin{equation}\label{centralizer}
\left \|\Omega_{\Psi,\Phi}(x+y) - \Omega_{\Psi,\Phi} (x) - \Omega_{\Psi,\Phi} (y) \right\|_\Phi
\leq 2(1+\varepsilon)\|\Psi: \ker \Phi\to \Psi(\ker \Phi)\| \left(\|x\|_\Phi + \|y\|_\Phi\right).
\end{equation}
Note also that $\|(\cdot,\cdot)\|_{\Omega_{\Psi,\Phi}}$ depends on the choice made defining $B_\Phi$, but
$d\Omega_{\Psi,\Phi}$ does not.
\end{remark}

We get an exact sequence
\begin{equation}\label{exact-seq-bis}
\begin{CD}
0 @>>> \Psi(\ker\Phi) @>j>> d\Omega_{\Psi,\Phi}  @>q>> X_\Phi @>>>0
\end{CD}
\end{equation}
with inclusion $jw =(w,0)$ and quotient map $q(w,x)= x$.

\begin{prop}\label{isomorphism}
The lower row of diagram (\ref{psidiagram}) is an exact sequence equivalent to (\ref{exact-seq-bis}).
In particular $X_{\Psi,\Phi}$ is isomorphic to $d\Omega_{\Psi,\Phi}$.
\end{prop}
\begin{proof}

We will show that the identity operator  $(w,x) \to (w, x)$ makes commutative the diagram
\begin{equation}\label{equiv-seq}
\begin{CD}
0 @>>> \Psi(\ker\Phi) @>i>> d\Omega_{\Psi,\Phi}  @>p>> X_\Phi @>>>0\\
&&@| @VVV @|\\
0 @>>> \Psi(\ker\Phi) @>j>> X_{\Psi,\Phi}  @>q>> X_\Phi @>>>0\end{CD}
\end{equation}
showing that the rows are equivalent exact sequences.

Fix $C \geq 1$ and let $(w,x)\in d\Omega_{\Psi,\Phi}$.
Since $w-\Omega_{\Psi,\Phi} x\in \Psi(\ker\Phi)$, $w - \Omega_{\Psi,\Phi} x = \Psi f$ for some
$f \in \ker \Phi$ with $\|f\|_{\mathcal H} \leq C \|w - \Omega_{\Psi,\Phi} x\|_q$.
Thus $w = \Omega_{\Psi,\Phi} x + \Psi f = \Psi( B_{\Phi} x + f)$ and therefore
$(w,x) = (\Psi( B_{\Phi} x + f), \Phi( B_{\Phi} x + f))\in X_{\Psi,\Phi}$ with
\begin{eqnarray*}
\|(w,x)\| &\leq&  \|B_{\Phi} x + f\| \leq (1+\varepsilon)\|x\| + C\|w - \Omega_{\Psi,\Phi} x\|
\leq (1+\varepsilon+C)\|(w,x)\|_{\Omega_{\Psi, \Phi}}.
\end{eqnarray*}
\end{proof}

Next we follow the concluding sections of \cite{cabenon} by studying the domain and range
spaces associated to $\Omega_{\Psi, \Phi}$. See also \cite{caceso} for a similar analysis.

\begin{defin}\label{def:Dom-Ran-bis}
Let $(\Psi, \Phi)$ be a pair of interpolators on $\mathcal{H}$.
The \emph{domain} and \emph{range} of $\Omega_{\Psi, \Phi}$ \emph{with respect to the exact sequence
(\ref{exact-seq-bis})} are defined as follows:
$$
\Dom(\Omega_{\Psi,\Phi})= \{x\in X_\Phi: \Omega_{\Psi, \Phi}(x)\in \Psi(\ker\Phi)\}
$$
endowed with the quasi-norm\;
$\|x\|_{\Dom(\Omega_{\Psi, \Phi})} = \|\Omega_{\Psi, \Phi} x\|_q + \|x\|_\Phi$ and
$$
\Ran(\Omega_{\Psi,\Phi})= \{w\in\Sigma: \exists x\in X_\Phi,\; w-\Omega_{\Psi,\Phi}(x)\in \Psi(\ker\Phi)\}
$$
endowed with\; $\|w\|_{\Ran(\Omega_{\Psi, \Phi})} = \inf \{\|w - \Omega_{\Psi, \Phi}(x)\|_q +
\|x\|_\Phi : x\in X_\Phi,\; w - \Omega_{\Psi, \Phi}(x)\in \Psi(\ker\Phi)\}$.
\end{defin}

Observe that $\Ran(\Omega_{\Psi, \Phi})$ is different from the set $Rang(\Omega_A)$ considered in
\cite[Definition 10]{caceso}, which is not a vector space in general.
Here we follow \cite{cabenon}.

\begin{prop}\label{dom-ran}
The maps $J x= (0,x)$ and $Q(w,y)=w$ define an exact sequence
\begin{equation}\label{eq:seq-DomRan}
\begin{CD}
0@>>> \Dom(\Omega_{\Psi,\Phi}) @>J>> d\Omega_{\Psi,\Phi} @>Q>> \Ran(\Omega_{\Psi,\Phi}) @>>>0.
\end{CD}
\end{equation}
with $\|J x\|_{\Omega_{\Psi,\Phi}}=\|x\|_\Dom$ and
$\|w\|_\Ran =\inf\{\|(w,x)\|_{\Omega_{\Psi,\Phi}}: (w,x)\in d\Omega_{\Psi, \Phi} \}$.
\end{prop}
\begin{proof}
Note that $x\in \Dom(\Omega_{\Psi, \Phi})$ if and only if $(0,x)\in d\Omega_{\Psi,\Phi}$,
and $\|x\|_{\Dom(\Omega_{\Psi, \Phi})} =  \|(0,x)\|_{\Omega_{\Psi,\Phi}}$.
Therefore $J$ is an ``isometric" operator.

The image of $J$ is closed: if a sequence $(0,x_n)$ in $\textrm{Im}(J)$ converges to
$(y,x)\in d\Omega_{\Psi,\Phi}$, then $\lim_n x-x_n= 0$  and $\lim_n y -\Omega_{\Psi,\Phi}(x-x_n)=0$
in $X_{\Phi}$.
Since $\Omega_{\Psi, \Phi}:X_\phi\to\Sigma$ is continuous at $0$ (because $B_\Phi$ is so) and the
inclusion $X_\Phi\to\Sigma$ is continuous, $y=0$.

The map $Q$ is well-defined and surjective:
$(w,x)\in d\Omega_{\Psi,\Phi}$ implies $w-\Omega_{\Psi,\Phi}x\in X_\Phi$, hence
$w\in \Ran(\Omega_{\Psi,\Phi})$.
Moreover, $w\in \Ran(\Omega_{\Psi,\Phi})$ implies the existence of $x\in X_\Phi$ such that
$w-\Omega_{\Psi,\Phi}\in X_\Phi$, hence $(w,x)\in d\Omega_{\Psi,\Phi}$.

Also it is clear that $\textrm{Im}(J)=\ker (Q)$ and that $\|w\|_{\Ran(\Omega_{\Psi, \Phi})}$ satisfies the
required equality.
\end{proof}

\begin{corollary} \cite{cabenon}
The spaces $\Dom(\Omega_{\Psi,\Phi})$ and $\Ran(\Omega_{\Psi,\Phi})$, endowed with their
respective quasi-norms, are complete.
\end{corollary}

The exact sequences (\ref{exact-seq-bis}) and (\ref{eq:seq-DomRan}) provide two different
representations of the derived space $d\Omega_{\Psi,\Phi}$ as a twisted sum.
See Examples \ref{ex:1} and \ref{ex:2}. The following result shows that the relation between these two exact sequences is symmetric.

\begin{prop}\label{prop:dom-ran}
Let $(\Psi,\Phi)$ be a pair of interpolators on $\mathcal H$. Then
\begin{enumerate}
  \item $\Dom(\Omega_{\Psi,\Phi}) =\Phi(\ker\Psi)$.
  \item $\Ran(\Omega_{\Psi,\Phi}) = X_\Psi$.
  \item The derivation associated to the exact sequence (\ref{eq:seq-DomRan}) is $\Omega_{\Phi,\Psi}$.
\end{enumerate}
\end{prop}
\begin{proof}
(1) If $x\in \Dom(\Omega_{\Psi,\Phi})$ then $x\in X_\Phi$ and $\Psi B_\Phi x\in \Psi(\ker \Phi)$.
Thus $\Psi B_\Phi x=\Psi g$ for some $g\in \ker \Phi$, hence $B_\Phi x- g\in \ker \Psi$ and
$x=\Phi(B_\Phi x- g)\in \Phi(\ker \Psi)$.
Conversely, if $y \in \Phi(\ker \Psi)$ then $y \in X_{\Phi}$ and there is $f \in \ker \Psi$ such that $y = \Phi(f)$. We have $B_{\Phi}(y) - f \in \ker \Phi$, so $\Psi(B_{\Phi}(y) - f) = \Omega_{\Psi, \Phi}(y) \in \Psi(\ker\Phi)$. Hence $y \in \Dom(\Omega_{\Psi, \Phi})$.

(2) If $w\in\Ran(\Omega_{\Psi,\Phi})$ then there exists $x\in X_\Phi$ such that
$w-\Omega_{\Psi,\Phi}x\in \Psi(\ker\Phi)\subset X_\Psi$.
Since $\Omega_{\Psi,\Phi}x\in X_\Psi$, we get $w\in X_\Psi$.
Conversely, if $w\in X_\Psi$ then $w=\Psi f$ for some $f\in\mathcal{H}$.
Since $(\Psi f,\Phi f)\in X_{\Psi,\Phi}$, Proposition \ref{dom-ran} implies
$w=\Psi f\in \Ran(\Omega_{\Psi,\Phi})$.

(3) We have to show that
$X_{\Phi,\Psi}=\{(w,x)\in\Sigma\times X_\Psi : w-\Omega_{\Phi,\Psi}x\in\Phi(\ker\Psi)\}$.
If $f\in\mathcal{H}$, then $\Psi f\in X_\Psi$ and
$\Phi f- \Omega_{\Phi, \Psi}\Psi f =\Phi(f-B_\Psi \Psi f)\in \Phi(\ker\Psi)$ because
$f-B_\Psi \Psi f\in \ker\Psi$.

Conversely, if $x\in X_\Psi$ and $w-\Omega_{\Phi,\Psi}x\in\Phi(\ker\Psi)$ then $x=\Psi h$
and $w-\Phi B_\Psi x=\Phi g$ with $h\in\mathcal{H}$ and $g\in \ker\Psi$.
Thus $\Phi(g+B_\Psi x)=w$ and $\Psi(g+B_\Psi x)=\Psi B_\Psi x =x $, hence
$(w,x)\in X_{\Phi, \Psi}$.
\end{proof}

Thus, the general situation can be described by the diagram

$$
\xymatrix{& \Ran(\Omega_{\Psi, \Phi})\ar@/^2pc/[rd]^{\Omega_{\Phi, \Psi}} &\\
 \Dom(\Omega_{\Phi, \Psi})\ar[r]\ar@/^1.3pc/[dr]_{\Omega_{\Phi,\Psi}}& X_{\Psi, \Phi}\ar[r]\ar[u]
 &\Ran(\Omega_{\Phi, \Psi})\ar@/^1.3pc/[ul]_{\Omega_{\Psi,\Phi}}\\
 &\Dom(\Omega_{\Psi, \Phi})\ar@/^2pc/[lu]^{\Omega_{\Psi,\Phi}}\ar[u]}
$$

In the language of \cite{cabenon}, if $\Omega=\Omega_{\Psi,\Phi}$ then $\mho=\Omega_{\Phi,\Psi}$.
The claim ``the roles of $\mho$ and $\Omega$ are perfectly symmetric" in \cite[p. 48]{cabenon} refers to the Kalton-Peck case, the compatible situation described in Section \ref{compatible} in which $\Dom(\Omega_{\Phi,\Psi})= X_\Phi= \Ran(\Omega_{\Phi,\Psi})$. In this case the preceding diagram becomes:
$$
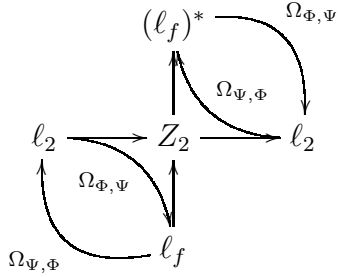
\xymatrix{& (\ell_f)^*\ar@/^2pc/[rd]^{\Omega_{\Phi, \Psi}} &\\
 \ell_2\ar[r]\ar@/^1.3pc/[dr]_{\Omega_{\Phi,\Psi}}& Z_2\ar[r]\ar[u]
 &\ell_2\ar@/^1.3pc/[ul]_{\Omega_{\Psi,\Phi}}\\
 &\ell_f \ar@/^2pc/[lu]^{\Omega_{\Psi,\Phi}}\ar[u]}
$$

\section{The bounded splitting theorem}

We study now the splitting and bounded splitting of the induced sequences. Recall that he exact sequence
\begin{equation}
\begin{CD}
0 @>>> \Psi(\ker\Phi) @>>> d\Omega_{\Psi,\Phi}  @>>> X_\Phi @>>>0
\end{CD}
\end{equation}
generated by the differential
$\Omega_{\Psi, \Phi}$ boundedly split if $\Omega_{\Psi, \Phi}: X_\Phi\to  \Psi(\ker\Phi)$ is bounded. The sequence splits if there is a linear map $L: X_\Phi\to \Sigma$ such that $\Omega_{\Psi, \Phi} - L: X_\Phi\to  \Psi(\ker\Phi)$ is bounded. The following result extends and completes \cite[Theorem 3.16]{ccfg}.

\begin{theorem}\label{main}
For a pair $(\Psi, \Phi)$ of interpolators, the following conditions are equivalent:
\begin{enumerate}
\item $\mathcal H = \ker\Phi+ \ker\Psi$.
\item $X_\Phi = \Phi(\ker \Psi)$.
\item $X_\Psi = \Psi(\ker \Phi)$.
\item $\Dom(\Omega_{\Psi,\Phi})= X_\Phi$.
\item $\Dom(\Omega_{\Phi, \Psi})= X_\Psi$.
\end{enumerate}
The above conditions are also equivalent to their ``topological" counterparts:
\begin{enumerate}
\item[(1')] $\mathcal H =\ker\Phi +\ker\Psi$ and there exists $C>0$ such that for every $f\in\mathcal H$
we can find $g\in\ker\Psi$ and $h\in\ker\Phi$ with $f=g+h$, $\|g\|\leq C\|f\|$ and $\|h\|\leq C\|f\|$.
\item[(2')] $X_\Phi = \Phi(\ker \Psi)$ with equivalent norms.
\item[(3')] $X_\Psi = \Psi(\ker \Phi)$ with equivalent norms.
\item[(4')] $\Omega_{\Psi, \Phi}$ is bounded from $X_\Phi$ to $\Psi(\ker\Phi)$.
\item[(5')] $\Omega_{\Phi, \Psi}$ is bounded from $X_\Psi$ to $\Phi(\ker\Psi)$.
\end{enumerate}
\end{theorem}
\begin{proof}
(1) $\Leftrightarrow$ (2) because $\Phi(\mathcal H)= \Phi(\ker \Psi)$ if and only if
$\mathcal H = \ker\Phi+ \ker\Psi$.
Similarly (1) $\Leftrightarrow$ (3).

Clearly (3) $\Rightarrow$ (4) and, by Proposition \ref{prop:dom-ran}, (4) $\Rightarrow$ (3).
Similarly (2) $\Leftrightarrow$ (5).
\medskip

(1) $\Rightarrow$ (1') Since $\mathcal H = \ker\Phi+ \ker\Psi$, the map $\Psi:\ker\Phi\to X_\Psi$ is open.
Thus there exists $c>0$ such that, for every $f\in \mathcal H$, we can find $g\in \ker\Phi$ with
$\|g\|\leq c\|f\|$ and $\Psi f=\Psi g$.
Then $f-g \in \ker\Psi$, $\|f-g\|\leq (1+c)\|f\|$ and $f =g + (f-g)$.

(2) $\Rightarrow$ (2') follows from $\dist(f, \ker\Phi)\leq \dist(f, \ker\Phi\cap\ker\Psi)$
and the open mapping theorem, and the proof of (3) $\Rightarrow$ (3') is similar.

(3') $\Rightarrow$ (4') is a consequence of
$\|\Omega_{\Psi, \Phi}f\|_\Psi=\|\Psi B_\Phi f\|_\Psi \leq \|\Psi\|(1+\varepsilon)\|f\|_\Phi$,
(4') $\Rightarrow$ (3') follows from Proposition \ref{prop:dom-ran}, and the proof of
(2') $\Leftrightarrow$ (5') is similar.
\end{proof}

In particular, Theorem \ref{main} shows that the sequence
\begin{equation}
\begin{CD}
0 @>>> \Psi(\ker\Phi) @>>> d\Omega_{\Psi,\Phi}  @>>> X_\Phi @>>>0
\end{CD}
\end{equation}
boundedly splits precisely when $\mathcal H=\ker\Phi+ \ker\Psi$, which also happens if an only if
\begin{equation}
\begin{CD}
0 @>>> \Phi(\ker\Psi) @>>> d\Omega_{\Psi,\Phi}  @>>> X_\Psi @>>>0
\end{CD}
\end{equation}
boundedly splits. Thus, it also shows the symmetric role of the interpolators:
\begin{corollary}
$\Omega_{\Psi,\Phi}$ is bounded if and only if so is $\Omega_{\Phi,\Psi}$.
\end{corollary}

\begin{prob} Is it true that $\Omega_{\Psi,\Phi}$ is trivial if and only if so is $\Omega_{\Phi,\Psi}$?
\end{prob}

\section{Examples}

Several relevant examples in the literature admit a formulation in the schema of pairs we have just presented. They include
Cwikel, Kalton, Milman, Rochberg differential methods \cite{ckmr}, the compatible and almost compatible interpolators of Carro, Cerd\`a and Soria \cite{caceso}, see Section \ref{compatible}, the translation operators considered by Cwikel, Jawerth, Milman and Rochberg \cite{cjmr} and, of course, the complex and real methods.

\subsection{Differential methods of Cwikel, Kalton, Milman, Rochberg}\label{ex:ckmr}
The so-called differential methods of Cwikel, Kalton, Milman and Rochberg \cite{ckmr} correspond to our schema
of two interpolators $(\Psi, \Phi)$, and this is the content of \cite[Section 5]{ckmr}.
With the same notation used there (see \cite{ckmr} for precise definitions): $\overline B = (X_0, X_1)$, $\textbf X = (\mathscr X_0, \mathscr X_1)$
is a couple of (Laurent compatible) pseudolattices, and
$$
\mathscr J(\textbf X, \overline B) = \{ (b_n)_{n\in \mathbb Z}: b_n\in X_0\cap X_1,
(e^{jn}b_n)_{n\in \mathbb Z}\in \mathscr X_j(B_j), j=0,1\}
$$
is endowed with the norm $\|(b_n)\| = \max_{j=0,1} \|(e^{jn}b_n)\|_{ X_j(B_j)}$.
The space of parameters $D$ is the open annulus $\mathbb A = \{z\in \mathbb C: 1<|z|<e\}$. The Laurent compatibility allows
one to identify the elements of $\mathscr J(\textbf X, \overline B)$ with certain analytic functions $f: \mathbb A \to X_0+X_1$.
For $s \in \mathbb A$, two interpolators $\Phi_s: \mathscr J(\textbf X, \overline B)\rightarrow \Sigma$
and  $\Psi_s: \mathscr J(\textbf X, \overline B)\rightarrow \Sigma$ are given by
$$
\Phi_s((b_n) ) = \sum s^n b_n\quad \textrm{and}\quad  \Psi_s((b_n) ) = \sum ns^{n-1} b_n.
$$

It is carefully shown in \cite{ckmr} that these methods subsume most versions of the real and complex
interpolation methods \cite[Section 4]{ckmr}; and also the method of compatible and almost-compatible interpolators of Carro, Cerd\`{a} and Soria \cite{caceso} to which we will return later (see \cite[Section 5]{ckmr} (in particular, the differential condition \cite[Def. 3.4]{ckmr} is there to get almost-compatible interpolators, while an additional condition (the left-shift maps boundedly $\mathscr J(\textbf X, \overline B)$ into itself) is required to make the interpolator compatible).

\subsection{Translation operators}\label{trans}

Consider two evaluation interpolators $(\Phi_\theta, \Phi_\nu)$
associated to a differential interpolation method as above. The associated differential $\Phi_\theta B_{\Phi_\nu}$ is the translation map $\mathscr R_{\theta, \nu}$ considered in either \cite{ckmr} or \cite{cjmr}. In this case the symmetric differential is
obviously $\mathscr R_{\nu, \theta}$.  Observe that $\mathscr R_{\nu, \theta}: X_\theta \to X_\nu$ is clearly bounded so the induced exact sequence splits. Slightly less obvious is that also the natural sequence generated by $\mathscr R_{\nu, \theta}$
$$\begin{CD} 0@>>> \Phi_\nu(\ker \Phi_\theta)@>>> X_{\nu, \phi} @>>> X_\theta@>>> 0\end{CD}$$
splits: this is consequence of Theorem \ref{main} and:

\begin{lemma} $\mathcal H = \ker \Phi_\nu + \ker \Phi_\theta$.\end{lemma}
\begin{proof} Let $\varphi$ be a conformal map from the annulus to the disk so that $\varphi(\nu)=0$ and let $f\in \mathcal H$. Then $f = \frac{\varphi}{\varphi(\theta)}f + f - \frac{\varphi}{\varphi(\theta)}f$.
\end{proof}

\begin{corollary}\label{equiva} $\Phi_\nu(\ker \Phi_\theta) = X_\nu$ with equivalence of norms.
\end{corollary}

\subsection{The Kalton-Peck space}\label{ex:1}
We consider the couple $(\ell_1,\ell_\infty)$ and the compatible (see Section \ref{compatible})  pair of interpolators $\Psi=\delta_{1/2}'$
and $\Phi=\delta_{1/2}$ from complex interpolation.
Then $X_\Phi=\ell_2$, $\Omega_{\Psi,\Phi}(x) =2x\log |x|/\|x\|_2$ and $X_{\Psi,\Phi} = Z_2$, the
Kalton-Peck space. According to \cite{kaltpeck}, $\Dom(\Omega_{\Psi, \Phi})= \ell_f$ is the Orlicz sequence space
generated by $f(t)=t^2 \log^2 t$ and $\Ran(\Omega_{\Psi, \Phi})= \ell_f^*$, its Orlicz dual pace. Therefore the map $\Omega_{\Phi, \Psi}$ is related to Lambert's $W$ function (the inverse of $z \mapsto z e^{z}$), since $f^{-1}(t) = -\frac{\sqrt{t}}{W(\sqrt{t})}$ for $t$ small enough.

\subsection{Weigthed K\"othe spaces}\label{ex:2}
Fix a K\"othe function space $X$ with the Radon-Nikodym property, let $w_0$ and $w_1$ be weight functions,
and consider the interpolation couple $(X_0, X_1)$, where $X_j= X(w_j)$, $j =0,1$ with their natural norms.
In \cite[Proposition 4.1]{ccfg} we showed that $X_\theta = X(w_\theta)$ for $0<\theta <1$, where
$w_{\theta} = w_0^{1 - \theta}w_1^{\theta}$.
For $\Psi = \delta'_{\theta}$ and $\Phi = \delta_{\theta}$ we obtain
$\Omega_{\Psi, \Phi}f = \log \frac{w_1}{w_0}\cdot f$, a linear map.

Let us determine $\Dom(\Omega_{\Psi, \Phi})$ and $\Ran(\Omega_{\Psi, \Phi})$:
\smallskip

\noindent
\emph{Claim 1: $\Dom(\Omega_{\Psi, \Phi}) = X(w_\theta) \cap X(w_\theta \left|\log \frac{w_1}{w_0}\right|)$
with equivalence of norms.}

Indeed, $x \in \Dom(\Omega_{\Psi,\Phi})$ if and only if $x$ and $\Omega_{\Psi,\Phi}(x)=\log\frac{w_1}{w_0}x$
belongs to $X(w_{\theta})$.
\smallskip

\noindent
\emph{Claim 2: $\Ran(\Omega_{\Psi, \Phi})= X(w_\theta)+ X(w_\theta \left|\log \frac{w_1}{w_0}\right|^{-1})$
with equal norms.}

If $w \in\Ran(\Omega_{\Psi,\Phi})$, then we may write $w = w-\log\frac{w_1}{w_0}x +\log\frac{w_1}{w_0}x$
with $w - \log \frac{w_1}{w_0} x \in X(w_{\theta})$ and $x \in X(w_{\theta})$.
Then $\log \frac{w_1}{w_0} x \in X(w_\theta \left|\log \frac{w_1}{w_0}\right|^{-1})$, hence
$w \in X(w_\theta) + X(w_\theta \left|\log \frac{w_1}{w_0}\right|^{-1})$ and
\[
\|w\|_{X(w_\theta)+ X(w_\theta\left|\log \frac{w_1}{w_0}\right|^{-1})}\leq \|w\|_{\Ran(\Omega_{\Psi,\Phi})}.
\]

If $w \in X(w_\theta) + X(w_\theta \left|\log \frac{w_1}{w_0}\right|^{-1})$, then $w=y+z$ with
$y \in X(w_\theta)$ and $z \in X(w_\theta \left|\log \frac{w_1}{w_0}\right|^{-1})$.
Also, there is $x \in X(w_{\theta})$ such that $z = \log \frac{w_0}{w_1} x$ and
$\|z\|_{X(w_\theta \left|\log \frac{w_1}{w_0}\right|^{-1})} = \|x\|_{X(w_{\theta})}$.
So we get the other inclusion and the other norm estimate.
\smallskip

To finish the description of $\Dom(\Omega_{\Psi, \Phi})$ and $\Ran(\Omega_{\Psi, \Phi})$, let us denote
$\omega_\wedge = \min\{\omega_0, \omega_1\}$ and $\omega_\vee = \max\{w_0, w_1\}$.
\smallskip

\noindent
\emph{Claim 3: $X(\omega_0) \cap X(\omega_1) = X(\omega_\vee)$ with equivalence of norms.}

Let $x \in X(\omega_\vee)$. Then $\max\{\|\omega_0 x\|_X, \|\omega_1 x\|_X\} \leq \|=
\max\{\|x\|_{X(\omega_0)}, \|x\|_{X(\omega_1)}\} \leq \|x\|_{\omega_\vee}$.

If $x \in X(\omega_0) \cap X(\omega_1)$ and $A$ is the set where $w_0 \leq w_1$ then
$\omega_\vee x = w_1 x \chi_A + w_0 x (1 - \chi_A)$, so
$$
\|\omega_\vee x\|_X \leq \|w_1 x\|_X + \|w_0 x\|_X \leq 2 \|x\|_{X(\omega_0) \cap X(\omega_1)}
$$
and we get the reverse inclusion.
\smallskip

\noindent
\emph{Claim 4: $X(\omega_0) + X(\omega_1) = X(\omega_\wedge)$ with equivalence of norms}

Let $x = x_0 + x_1 \in X(\omega_0) + X(\omega_1)$ with $x_j \in X(\omega_j)$, $j = 0, 1$. We have
\begin{eqnarray*}
\|\omega_\wedge x\|_X & \leq & \|\omega_\wedge x_0\|_X + \|\omega_\wedge x_1\|_X
               \leq  \|\omega_0 x_0\|_X + \|\omega_1 x_1\|_X
                =  \|x_0\|_{X(\omega_0)} + \|x_1\|_{X(\omega_1)}.
\end{eqnarray*}
Since $x_0$ and $x_1$ are arbitrary, $\|x\|_{X(\omega_\wedge)}\leq \|x\|_{X(\omega_0)+ X(\omega_1)}$.

Now let $x \in X(\omega_\wedge)$, and let $A$ be as above. Let $x_0 = x \chi_A$ and $x_1 = x (1-\chi_A)$.
Then $x_j\in X(\omega_j)$, $j=0,1$, and
$\|x_0\|_{X(\omega_0)}+\|x_1\|_{X(\omega_1)}\leq 2\|x\|_{X(\omega_\wedge)}$,
so we obtain the reverse inclusion.\smallskip

In order to calculate $\Omega_{\Phi, \Psi}$, recall that $X_{\Psi} = \Ran(\Omega_{\Psi, \Phi})$.
If we let $\omega = \min\{1, \left|\log \frac{w_1}{w_0}\right|^{-1}\} w_{\theta}$ then by Claims 2 and 4
we have $X_{\Psi} = X(\omega)$ with equivalence of norms.
Let $x \in X_{\Psi}$ and let us suppose that $\min\{1, \left|\log \frac{w_1}{w_0}\right|^{-1}\} =
(\log \frac{w_1}{w_0})^{-1}$.
Then $(\log \frac{w_1}{w_0})^{-1} x \in X(\omega_{\theta}) = X_{\Phi}$, and the function $B_{\theta}(x)(z) = \Big(\frac{w_1}{w_0}\Big)^{z - \theta} (\log \frac{w_1}{w_0})^{-1}x$
is in $\mathcal H$, its norm is $\|(\log\frac{w_1}{w_0})^{-1} x\|_{X(\omega_{\theta})} =\|x\|_{X(\omega)}$
and $\Psi(B_{\theta}(x)) = x$.
So $\Omega_{\Phi, \Psi} = \Phi(B_{\theta}(x)(z)) = (\log \frac{w_1}{w_0})^{-1} x$.\medskip

\section{Compatibility-like conditions}\label{compatible}

Some special pairs of interpolators were studied in \cite{caceso}:

\begin{defin}
A pair of interpolators $(\Psi,\Phi)$ on the same space $\mathcal H$ is called \emph{almost compatible}
when $\Psi (\ker \Phi)\subset X_\Phi$, and it is called \emph{compatible} when $\Psi (\ker \Phi)= X_\Phi$.
\end{defin}

In the case of compatible pairs $(\Psi,\Phi)$ ,
the exact sequence (\ref{exact-seq-bis}) becomes
\begin{equation}\label{exact-seq-compat}
\begin{CD}
0 @>>> X_\Phi @>>> d\Omega_{\Psi,\Phi}  @>>> X_\Phi @>>>0.
\end{CD}
\end{equation}
So $d\Omega_{\Psi,\Phi}$ is a twisted sum of $X_\Phi$ with itself. When $(\Psi, \Phi)$ is almost compatible, since the inclusion map $X_\Phi\to\Sigma$
is continuous, the map $\Psi :\ker \Phi\to X_\Phi$ is continuous by the closed graph theorem.
Similarly, when $(\Psi, \Phi)$ is compatible, the inclusion $X_\Phi\to X_\Psi$ is continuous.
\medskip

Let us now consider what occurs when the same differential is used to generate twisted sums with
larger spaces. This is interesting to cover the case of almost compatible interpolators.

\begin{defin}
Let $(\Psi, \Phi)$ be a pair of interpolators on $\mathcal{H}$.
We say that a subspace $Z$ of $\Sigma$ is \emph{suitable for $(\Psi, \Phi)$} if $\Psi(\ker\Phi)\subset Z$
and there is a norm $\|\cdot\|_Z$  on $Z$ such that $(Z,\|\cdot\|_Z)$ is a Banach space and the inclusion
$(Z,\|\cdot\|_Z)\to\Sigma$ is continuous.
\end{defin}

The derivation $\Omega_{\Psi,\Phi}$ and a suitable space $Z$ generate a derived space
$$
d\Omega_{\Psi,\Phi}(Z):=\dPsiZ =\{(w,x)\in\Sigma\times X_\Phi: w-\Omega_{\Psi,\Phi}x\in Z\},
$$
endowed with $\|(w,x)\|_{\Omega^Z_{\Psi,\Phi}}= \|w - \Omega_{\Psi,\Phi} x\|_Z + \|x\|_\Phi$,
which can be showed to be a quasi-norm arguing as in Remark \ref{rem:qnorm}.
We also obtain an exact sequence
\begin{equation}\label{exact-seq-Z}
\begin{CD}
0 @>>> Z @>>> d\Omega_{\Psi,\Phi}(Z) @>>> X_\Phi @>>>0
\end{CD}
\end{equation}
with inclusion $w\to (w,0)$ and quotient map $(w,x)\to x$.
\medskip

The case $(\Psi, \Phi)$ almost compatible and $Z=X_\Phi$ was studied in \cite{caceso}.

\begin{defin}\label{def:Dom-Ran}
Let $(\Psi, \Phi)$ be a pair of interpolators on $\mathcal{H}$, and let $Z$ be a suitable space.
We define the \emph{domain} and the \emph{range} of $\Omega_{\Psi, \Phi}$
\emph{with respect to the exact sequence (\ref{exact-seq-Z})} as follows:
$$
\Dom(\Omega^Z_{\Psi, \Phi})= \{x\in X_\Phi: \Omega_{\Psi, \Phi}(x)\in Z\}
$$
endowed with\;
$\|x\|_{\Dom(\Omega^Z_{\Psi, \Phi})} = \|\Omega_{\Psi, \Phi} x\|_Z + \|x\|_\Phi$, and
$$
\Ran(\Omega^Z_{\Psi, \Phi})= \{w\in \Sigma : \exists x\in X_\Phi,\; w - \Omega_{\Psi, \Phi}(x)\in Z\}
$$
endowed with\; $\|w\|_{\Ran(\Omega^Z_{\Psi, \Phi})} = \inf \{\|w - \Omega_{\Psi, \Phi}(x)\|_Z +
\|x\|_\Phi : x\in X_\Phi,\; w - \Omega_{\Psi, \Phi}(x)\in Z\}$.
\end{defin}

The arguments in the proof of Proposition \ref{dom-ran} give the following result:

\begin{prop}\label{dom-ran-Z}
Let $(\Psi, \Phi)$ be a pair of interpolators on $\mathcal{H}$, and let $Z$ be a suitable space.
Then the maps $J x= (0,x)$ and $Q(w,y)=w$ define an exact sequence
\begin{equation}\label{seq-DomRan}
\begin{CD}
0@>>> \Dom(\Omega^Z_{\Psi,\Phi}) @>J>> d\Omega_{\Psi,\Phi}(Z) @>Q>> \Ran(\Omega^Z_{\Psi,\Phi}) @>>>0.
\end{CD}
\end{equation}
with $\|J x\|_{\Omega^Z_{\Psi,\Phi}}=\|x\|_{\Dom(\Omega^Z_{\Psi, \Phi})}$ and
$\|w\|_{\Ran(\Omega^Z_{\Psi,\Phi})} =\inf\{\|(w,x)\|_{\Omega_{\Psi,\Phi}}: (w,x)\in d\Omega_{\Psi,\Phi}(Z)\}$.
\end{prop}
%
%
%

\begin{corollary}
The spaces $\Dom(\Omega^Z_{\Psi,\Phi})$ and $\Ran(\Omega^Z_{\Psi,\Phi})$, endowed with their
respective quasi-norms, are complete.
\end{corollary}

The following result gives a description of the domain in a non-symmetric case.

\begin{prop}\label{prop:dom-Z} \emph{\cite[Theorem 3.8]{caceso}}
Let $(\Psi,\Phi)$ be an almost compatible pair of interpolators on $\mathcal{H}$.
Then $\Dom(\Omega^{X_\Phi}_{\Psi,\Phi})= \Phi\left(\Psi^{-1}X_\Phi\right)$.
\end{prop}

It would be interesting to determine $\Ran(\Omega^{X_\Phi}_{\Psi,\Phi})$ for $(\Psi,\Phi)$
almost compatible.

\section{The general form of a commutator theorem}\label{sect:commutator}

Recall that given two maps $A,B$ in the suitable conditions, their commutator is defined as the map $[A,B]=AB-BA$. The purpose of this section is to show that there is just one commutator theorem from which all the existing versions can be derived.
Precisely:

\begin{theorem}[Abstract commutator theorem]\label{th:commutator}
Let $(X_0, X_1)$ be an interpolation couple of Banach spaces and let $(\Psi,\Phi)$ be a pair of interpolators
on $\mathcal H$. If $\tau$ is an operator on the scale then there is a commutative diagram
$$\begin{CD}\label{right}
0 @>>> \Dom (\Omega_{\Phi, \Psi}) @>>> X_{\Phi,\Psi} @>>> \Ran (\Omega_{\Phi, \Psi})  @>>>0\\
&&@V{\tau}VV @VV{(\tau, \tau)}V @VV\tau V\\
0 @>>> \Dom( \Omega_{\Phi, \Psi}) @>>> X_{\Phi,\Psi} @>>> \Ran (\Omega_{\Phi, \Psi}) @>>>0
\end{CD}$$
where $(\tau, \tau)$ represents the operator $(\tau, \tau)(w,x)=(\tau w, \tau x)$.
Equivalently, the commutator map $[\tau,\Omega_{\Psi,\Phi}]: \Ran(\Omega_{\Phi,\Psi})\to \Dom(\Omega_{\Phi,\Psi})$
is bounded and satisfies the following estimate
$$
\|[\tau, \Omega_{\Psi,\Phi}]\|\leq \max \left\{ \|\tau: \Psi(\ker \Phi)\to  \Psi(\ker \Phi)\|,
\|\tau: X_\Phi\to X_\Phi\|, 2\|T\| \|B_\Phi\|\right\}.
$$
\end{theorem}

\begin{proof} If the reader has been surprised to see $\Omega_{\Phi, \Psi}$ --the symmetric ``unknown" derivation--
observe that the result has been formulated so that the derivation that generates the exact sequences
is the ``well known" $\Omega_{\Psi, \Phi}$ and the commutative diagram above is exactly
$$
\begin{CD}
0 @>>> \Psi(\ker \Phi)  @>>>  X_{\Psi,\Phi} @>>> X_\Phi @>>>0\\
&&@V\tau VV @VV(\tau, \tau)V @VV\tau V\\
0 @>>> \Psi(\ker \Phi)  @>>> X_{\Psi,\Phi}  @>>> X_\Phi  @>>>0\end{CD}
$$
Observe that the statement makes sense since $x \in X_\Phi = \rm Ran(\Omega_{\Phi,\Psi})$ implies that $(\tau \Omega_{\Phi,\Psi} - \Omega_{\Phi,\Psi}\tau )x \in \Psi(\ker \Phi)$. Moreover, the operator $\tau: \Psi(\ker \Phi)\to \Psi(\ker \Phi)$ is well-defined and continuous:
if $x\in \Psi(\ker \Phi)$, i.e., $x=\Psi(g)$ with $\Phi(g)=0$ then $\Phi(Tg) = \tau(\Phi g) = 0$ and thus
$$
\tau(x) = \tau ( \Psi(g) ) = \Psi(Tg)) \in \Psi(\ker \Phi).
$$
The operator $(\tau, \tau):   X_{\Psi,\Phi}\to   X_{\Psi,\Phi}$ is well-defined: if $(w, x)\in X_{\Psi,\Phi}$,
namely $w-\Psi B_\Phi x\in \Psi(\ker\Phi)$ then  $w-\Psi B_\Phi x=\Psi g$ for some $g\in \ker \Phi$ and thus

\begin{eqnarray*}
\tau w - \Psi B_\Phi \tau x &=& \tau w - \Psi T B_\Phi x + \Psi T B_\Phi x - \Psi B_\Phi \tau x\\
&=& \tau w - \tau \Psi B_\Phi x  + \Psi T B_\Phi x - \Psi B_\Phi \tau x\\
&=& \tau (w - \Psi B_\Phi x ) + \Psi \left (T B_\Phi x - B_\Phi \tau x\right)\\
&=&\tau \Psi(g) + \Psi \left (T B_\Phi x - B_\Phi \tau x\right)\\
&=& \Psi (Tg)+  \Psi \left (T B_\Phi x - B_\Phi \tau x\right)
\end{eqnarray*}
which means that $\tau w -\Psi B_\Phi \tau x \in \Psi(\ker \Phi)$ since both $Tg$ and
$T B_\Phi x - B_\Phi \tau x$ belong to $\ker \Phi$ and thus $(\tau w, \tau x)\in X_{\Psi, \Phi}$.
It is continuous since, by the previous identity,
\begin{eqnarray*}
\|(\tau w, \tau x)\| &=&  \|\tau w - \Psi B_\Phi \tau x \|_{\Psi(\ker \Phi)} + \|\tau x\|_{\Phi}\\
  &\leq& \|\Psi T g\|_{\Psi(\ker \Phi)} + \|\Psi \left (T B_\Phi x - B_\Phi \tau x\right)
     \|_{\Psi(\ker \Phi)} + \|\tau x\|_{\Phi}\\
  &\leq& C\|(w, x)\| +  \|[\tau, \Omega_{\Psi,\Phi}](x)\|_{\Psi(\ker \Phi)},
\end{eqnarray*}
where $C=\max\left\{\|\tau:\Psi(\ker\Phi)\to  \Psi(\ker\Phi)\|, \|\tau: X_\Phi\to X_\Phi\|\right\}$, and
\begin{eqnarray*}
\|[\tau, \Omega_{\Psi,\Phi}](x)\|_{\Psi(\ker \Phi)}&=&\|\tau \Omega_{\Psi,\Phi} x -
    \Omega_{\Psi,\Phi} \tau x\|_{\Psi(\ker \Phi)}\\
&=&\|\tau \Psi B_\Phi (x) - \Psi B_\Phi (\tau x) \|_{\Psi(\ker \Phi)}\\
&=&\|\Psi T B_\Phi (x) - \Psi B_\Phi (\tau x) \|_{\Psi(\ker \Phi)}\\
&\leq&2\|T\| \|B_\Phi\| \|x\|.
\end{eqnarray*}
\end{proof}

When $(\Psi, \Phi)$ are compatible interpolators; i.\ e., $\Psi(\ker \Phi)= X_\Phi$ then the diagram in the
statement of Theorem \ref{th:commutator} adopts the more standard form:
$$\begin{CD}\label{comp}
0 @>>> X_\Phi  @>>> X_{\Psi,\Phi} @>>> X_\Phi @>>>0\\
&&@V{\tau}VV @VV{(\tau, \tau)}V @VV\tau V\\
0 @>>> X_\Phi @>>> X_{\Psi,\Phi} @>>> X_\Phi @>>>0\end{CD}
$$
and yields a recognizable estimate: the commutator map $[\tau, \Omega_{\Psi, \Phi}]: X_\Phi\to  X_\Phi$
is bounded. The symmetric version of Theorem \ref{th:commutator} has exactly
the same form just interchanging $\Psi$ and $\Phi$. Even if from the abstract point of view both theorems are ``the same" they may lead to quite different concrete estimates. A few examples follow:

\subsection{Commutator theorem for weighted spaces}
We refer to Example \ref{ex:2}. Let $X_0=X(w_0)$ and $X_1=X(w_1)$ be weighted versions of the same base space $X$.
In this case the commutator theorem for $\Omega = \Omega_{\delta', \delta} $ and its symmetric form for
$\mho= \Omega_{\delta, \delta'}$ are the same: call $a=\log\frac{w_1}{w_0}$ so that $\Omega(f)=af$ and let $min = \min\{1, \left|a\right|^{-1}\} w_{\theta}$ and $max = \max\{1, \left|a\right|\} w_{\theta}$. The continuity
of the commutator $[\tau, \Omega]: X ( w_\theta)  \to X (w_\theta)$ means
$$
\left\|w_\theta \left( \tau(af) - a\tau f \right) \right\|_X \leq \|w_\theta f\|_X
$$
while when $\mho g= a^{-1}g$, the continuity of $[\tau, \mho]: X ( min )  \to X (max )$ means
$$\left\|\tau (a^{-1}g) - a^{-1} \tau g \right\|_{max} \leq \|g\|_{min}$$

Thus, assuming $min = \left|a\right|^{-1}w_\theta$ and $max = w_\theta \left|a\right|$ then
$\left\|w_\theta a\left(\tau (a^{-1}g) - a^{-1} \tau g\right) \right\|_X \leq \|a^{-1}w_\theta g\|_X$ which,
by simple change of variable $g=af$, becomes, as we knew,

$$\left\|w_\theta a \left( \tau (f) - a^{-1} \tau (af) \right) \right\|_X =
\left\|w_\theta  \left( a\tau (f) - \tau (af) \right) \right\|_X  \leq \|w_\theta f\|_X.$$

\subsection{Commutator theorems for Lorentz spaces.}
Picking the couple $(L_{p_0}, L_{p_1})$ the derivation at $ \frac{1}{p} = (1-\theta)\frac{1}{p_0} + \theta\frac{1}{p_1}$
is the Kalton-Peck map $\mathscr{K}(x) = p(\frac{1}{p_0} - \frac{1}{p_1}) x \log \frac{|x|}{\|x\|_p}$,
and the standard commutator theorem means the estimate
$$
\left \|\tau\left (x\log \frac{|x|}{\|x\|_p}\right) - \tau(x) \log \frac{|\tau (x)|}{\|\tau(x)\|_p}\right \|_p \leq C \|x\|_p
$$
The symmetric commutator theorem means the estimate

$$\left \|\tau( \mho x) - \mho \left( \tau(x)\right) \right \|_{\ell_f} \leq C \|x\|_{\ell_f^*}$$

The more general version for Lorentz spaces $L_{p,q}$ is as follows.
Recall from \cite{cabe2} that $(L_{p_0, q_0}, L_{p_1, q_1})_\theta = L_{p,q}$ for
$\frac{1}{p} = (1-\theta)\frac{1}{p_0} + \theta\frac{1}{p_1}$ and $ \frac{1}{q} = (1-\theta)\frac{1}{q_0} + \theta\frac{1}{q_1}$
with derivation
$$
\Omega(x)=  q\left(\dfrac{1}{q_1}-\frac{1}{q_0}\right)\mathscr K(x)    +     \left(\frac{q}{p}\left(\dfrac{1}{q_0}-\frac{1}{q_1}\right)-\left(\dfrac{1}{p_0}-\frac{1}{p_1}\right)\right)\kappa(x)
$$
Here $\kappa$ denotes the Kalton map given by $\kappa(x) = x \; r_x$ where $r_x$ is the rank
function $r_x(t) = m\{s :|x(s)| >|x(t)|$ or $ |x(s)| =|x(t)|$ and $s \leq t\}$.
The case of $L_p$ spaces follows from this by setting $q_0=p_0$ and $q_1=p_1$. One thus gets the commutator estimate

$$\left \|\tau \Omega (f) - \Omega (\tau(x))\right \|_{p,q} \leq C \|f\|_{p,q}$$

\subsection{Commutator theorems for Orlicz spaces.}
Recall that an \emph{$N$-function} is a map $\varphi:[0,\infty)\ra[0,\infty)$ which is strictly
increasing, continuous, $\varphi(0)=0$, $\varphi(t)/t\ra 0$ as $t\ra 0$, and $\varphi(t)/t\ra \infty$
as $t\ra \infty$. An $N$-function $\varphi$ satisfies the \emph{$\Delta_2$-property} if there exists a number $C>0$
such that $\varphi(2t)\leq C\varphi(t)$ for all $t\geq 0$. When an $N$-function $\varphi$ satisfies the $\Delta_2$-property, the \emph{Orlicz space} $L_\varphi(\mu)$ is $L_\varphi(\mu) = \{f\in L_0(\mu) : \varphi(|f|)\in L_1(\mu)\}$ endowed with the norm
$\|f\|=\inf \{r>0 : \int \varphi(|f|/r) d\mu \leq 1\}.$

A combination of \cite{Gustavsson} and \cite{cfg} yields complex interpolation and the associated derivation. Given $\varphi_0$ and $\varphi_1$ two $N$-functions satisfying the $\Delta_2$-property and  $0<\theta<1$ then $\varphi^{-1} = \big(\varphi_0^{-1}\big)^{1-\theta} \big(\varphi_1^{-1}\big)^\theta$ is an $N$-function $\varphi$ satisfying the $\Delta_2$-property and $\big(L_{\varphi_0}(\mu), L_{\varphi_1}(\mu)\big)_\theta = L_\varphi(\mu)$. In particular, when $t = \varphi_0^{-1}(t) \varphi_1^{-1}(t)$ we have $\big(L_{\varphi_0}(\mu), L_{\varphi_1}(\mu)\big)_{1/2} = L_2(\mu)$ with associated derivation $
\Omega_{1/2}(f)= f\, \log \frac{\varphi_1^{-1}(f^2)}{\varphi_0^{-1}(f^2)}$ for $\|f\|_2 = 1$. The direct commutator estimate is thus
$$
\left \|\tau\left ( f \log \frac{ \|f\|_2 \varphi_1^{-1}(\frac{f^2}{\|f\|_2^2})}{f} \right) - \tau(f) \log \frac{ \|\tau(f)\|_2 \varphi_1^{-1}(\frac{\tau(f)^2}{\|\tau(f)\|^2_2})}{\tau(f)}\right \|_2 \leq C \|f\|_2
$$

The determination of the spaces $d\Omega_{\frac{1}{2}}$, $\rm Dom(\Omega_{\frac{1}{2}})$ and $\rm Ran(\Omega_{\frac{1}{2}})$ will be delayed to \cite{cabecor} since it requires a somewhat contorted digression into the theory of Fenchel-Orlicz spaces. The next example deals with a simpler case.

\subsection{Concavification and Fenchel-Orlicz spaces}

According to \cite{cfg}, if $X$ is a Banach space with $1$-unconditional basis which is $p$-convex and $X^p$ is the $p$-concavification of $X$ then $(\ell_{\infty}, X^p)_{\theta} = X$ and
\[
\Omega_{\theta}(x) = p x \log \frac{\left|x\right|}{\|x\|_X}
\]

Accordingly, the commutator estimate is
$$
\left \|\tau\left ( x \log \frac{\left|x\right|}{\|x\|_X} \right) - \tau(x) \log \frac{\left|\tau(x)\right|}{\|\tau(x)\|_X} \right \|_\theta \leq C \|x\|_\theta
$$

Suppose $X = \ell_{\varphi}$ where $\varphi$ is an $N-$function for which there are $p > 1$ and $M > 0$ such that for every $\lambda \in (0, 1]$ and for every $s > 0$ we have $\frac{\varphi(\lambda s)}{\lambda^p \varphi(s)} \leq M$. If $X$ has type greater than $1$, then we may suppose that $\varphi$ has the previous property. Now, $\Omega_{\theta}$ is a multiple of the quasilinear map defined on \cite{Fenchel}, so that $d\Omega_{\theta}$ is isomorphic to the quasi-normed Fenchel-Orlicz space $\ell_{\psi}$, where $\psi : \mathbb{C}^2 \rightarrow [0, \infty)$ is given by $\psi(x, y) = \varphi(y) + \varphi(x - y \log \left|x\right|)$. The isomorphism $T : d\Omega_{\theta} \rightarrow \ell_{\psi}$ is given by $T(x, y) = (x, py)$. Let $\xi(t) = \psi(0, t)$. Then $\rm Dom(\Omega_{\theta}) = \ell_{\xi}$. Since $\varphi(t) \leq \varphi(t \log \left|t\right|)$ on a neighborhood of 0, actually $\rm  Dom(\Omega_{\theta}) = \ell_{\varphi(t \log\left|t\right|)}$.


\subsection{Commutator theorems for translation operators}
Cwikel, Kalton, Milman, Rochberg obtain in \cite[Theorem 3.8 (ii)]{ckmr} a commutator theorem for translation mappings
which, as they say \cite[p.278]{ckmr}:
\begin{displayquote}
\emph{On the other hand, it is not a all clear to us at this stage how one could obtain a result like part (ii) of
Theorem 3.8 in the abstract setting of \cite{caceso}}.
\end{displayquote}
The result was integrated in the  schema of \cite{caceso} by Cerd\`{a} in \cite[p.1018]{cerda}; it is Proposition \ref{cerda} below.
The idea observed by Cerd\`{a} is to consider the pair of (non-compatible) interpolators $(\Phi_\theta, \Phi_\nu)$
associated to a differential interpolation method in the sense of \cite{ckmr} as described in Example \ref{ex:ckmr}.
In which case the differential $\Phi_\theta B_{\Phi_\nu}$ is the translation map $\mathscr R_{\theta, \nu}$ as decribed in Example \ref{trans}. We thus have:

\begin{prop}\label{trans}\emph{[Commutator/symmetric commutator theorem for translation maps]}
There is a commutative diagram
$$\begin{CD}
0 @>>> \Phi_\theta(\ker \Phi_\nu)  @>>>  X_{\theta,\nu} @>>> X_\nu @>>>0\\
&&@V\tau VV @VV(\tau, \tau)V @VV\tau V\\
0 @>>> \Phi_\theta(\ker \Phi_\nu)  @>>> X_{\theta,\nu}  @>>> X_\nu  @>>>0\end{CD}
$$
Equivalently, the commutator map $[\tau, \mathscr R_{\theta, \nu}]: X_\nu \to  \Phi_\theta(\ker \Phi_\nu)$ is bounded
\end{prop}


It is not necessary to formulate the symmetric form as we observed in Example \ref{trans}. Proposition \ref{trans} is not, in principle,  \cite[Theorem 3.8 (ii)]{ckmr} or \cite[Cor. 4.3]{cerda} since those results establish that the commutator map $[\tau, \mathscr R_{\theta, \nu}]:  X_\nu \to  X_\theta $ is bounded, as it is obvious since $\mathscr R_{\theta, \nu}:  X_\nu \to  X_\theta $ is bounded. However, by Corollary \ref{equiva},
$\Phi_{\theta}(\ker \Phi_{\nu}) = X_\theta$ and thus also  $\mathcal R_{\theta, \nu}:  X_\nu \to  \Phi_\theta(\ker \Phi_\nu)$ is bounded. The estimate one obtains in this form for the commutator is however more interesting:
\begin{prop}\label{cerda}
$$\left \|[\tau,  \mathscr R_{\theta, \nu}]: X_\nu\to \Phi_\theta(\ker \Phi_\nu)\right\|\leq g(\ker \Phi_\nu, \ker \Phi_\theta)$$
\end{prop}
\begin{proof}
From the last estimate in the proof of Theorem \ref{th:commutator} we get
\begin{eqnarray*}
\left\|[\tau,  \mathscr R_{\theta, \nu} ](x)\right\|_{\Phi_\theta(\mathcal H)} &\leq& \left\|\Phi_\theta T B_{\Phi_\nu}(x)
- \Phi_\theta B_{\Phi_\nu} \tau x \right\|_{\Phi_\theta(\mathcal H)} \\
&=&\mathrm{dist} \left(T B_{\Phi_\nu} (x) -  B_{\Phi_\nu} \tau x , \ker \Phi_\theta\right)\\
&\leq & \|T B_{\Phi_\nu} (x) - B_{\Phi_\nu} \tau x \|_{\mathcal H}\;\; g(\ker \Phi_\theta, \ker \Phi_\theta)\\
&\leq & 2\|T\|\| B_{\Phi_\nu}\|\|x\|_\nu \;\; g(\ker \Phi_\nu, \ker \Phi_\theta)\\
\end{eqnarray*}
since $T B_{\Phi_\nu} (x) - B_{\Phi_\nu} \tau x \in \ker \Phi_\nu$.
\end{proof}

This estimate is similar to \cite[Corollary 4.2]{cerda} although Cerd\`{a} uses an adaptation of the Krugljak-Milman
metric \cite{km} (see Section \ref{sect:stability}).

\section{Stability issues}\label{sect:stability}

We now enter into stability issues; namely, what occurs when passing from a pair of interpolators
$(\Psi_t, \Phi_t)$ to another one which is close in some sense.
We will consider the Krugljak-Milman metric \cite{km} $\rho(\Phi,\Psi)= \sup_{\|f\|\leq 1}\left|\Phi(f)-\Psi(f)\right|$
in the equivalent form
$$
g(\Phi, \Psi) = g(\ker \Phi, \ker \Psi).
$$
One has

\begin{prop}
Let $(\Psi, \Phi)$ be a pair of interpolators such that $\mathcal{H}=\ker\Psi+\ker\Phi$.
Then there exists $C>0$ such that if $(\Psi_1, \Phi_1)$ satisfies $g(\ker\Psi,\ker\Psi_1)<C$ and
$g(\ker\Phi,\ker\Phi_1)<C$ then $\mathcal{H}=\ker\Psi_1+\ker\Phi_1$.
\end{prop}
\begin{proof}
It is an application of Proposition \ref{prop:gap-stab}.
\end{proof}

\begin{lemma}\label{equal}
Let $(\Psi, \Phi)$ be two interpolators on $\mathcal H$. Then
$$
g(\Psi(\ker \Phi), \Phi(\ker \Psi)) = g(\ker \Phi, \ker \Psi).
$$
\end{lemma}
\begin{proof}
Here, as in Proposition \ref{isomorphism}, we identify $\Psi(\ker \Phi)$ and $\Phi(\ker \Psi)$ with the
subspaces $\ker\Phi/(\ker\Psi\cap\ker\Phi)$ and $\ker\Psi/(\ker\Psi\cap\ker\Phi)$ of
$X_{\Psi,\Phi}=\mathcal{H}/(\ker\Psi\cap\ker\Phi)$.
Moreover, given $f\in \mathcal{H}$, we denote $\tilde f= f+(\ker\Psi\cap\ker\Phi)\in X_{\Psi,\Phi}$.

For $f\in \ker\Psi$, we have
\begin{eqnarray*}
\dist\left(\tilde f,\Psi(\ker \Phi)\right) &=& \inf \{\|\tilde  f-\tilde g\| : \tilde g \in \Psi(\ker \Phi)\}\\
  &=& \inf \{\|f- g-h\| : g \in\ker \Phi, h\in \ker\Psi\cap\ker\Phi\}\\
  &=& \inf \{\|f- g\| : g \in\ker \Phi\}= \dist\left(f,\ker \Phi\right).
\end{eqnarray*}
Thus $\delta(\Phi(\ker \Psi),\Psi(\ker \Phi))= \delta(\ker\Phi, \ker \Psi)$, which implies the equality.
\end{proof}

Lemma \ref{equal} means that the interpolators are at the same distance no matter if acting on $\mathcal H$
or in $X_{\Psi,\Phi}$: observe that the quotient map $p: X_{\Psi, \Phi}\to X_\Phi$ in diagram (1) can be thought of as the map induced by $\Phi$ and then $\ker p = \ker \Phi_{|X_{\Psi, \Phi}} = \Psi(\ker \Phi)$.

\subsection{Continuous families}

\adef A family $(\Phi_d)_{d\in D}$ of interpolators $\mathcal H$ will be called continuous if
$$\lim_{t\to s} g(\Phi_t, \Phi_s) = 0.$$
A family of pairs $(\Psi_d, \Phi_d)_{d\in D}$ will be called \emph{bicontinuous} if $(\Phi_d)_{d\in D}$ is continuous and
$$
\lim_{t\to s} g\big(\ker \Psi_t \cap \ker \Phi_t, \ker \Psi_s \cap \ker \Phi_s\big)= 0.
$$
\zdef

Continuous interpolation methods immediately yield stability results:

\begin{prop}\label{prop:trans}
Let $(\Phi_d)_{d\in D}$ be a continuous family of interpolators on $\mathcal H$.
Let $\mathcal R_{t,s} = \Phi_t B_{\Phi_s} $ be the translation map.
Each $s\in D$ has a neighborhood $V$ such that for each $t\in V$ the differential
$\mathcal R_{t,s}$ is trivial if and only if $\mathcal R_{s,t}$ is trivial.
\end{prop}
\begin{proof}
Recall that ``$\mathcal R_{t,s}$ is trivial" means that the sequence $0\to \Phi_t(\ker \Phi_s)\to X_{t,s} \to X_s\to 0$
splits, so that $X_{t,s} = \Phi_t(\ker \Phi_s) \oplus N$ for some closed subspace $N$ of $X_{t,s}$ and thus
$\gamma (\Phi_t(\ker \Phi_s), N)>0$. Since the family of interpolators in continuous
$\lim_{t\to s} g(\ker \Phi_t, \ker \Phi_s) = \lim_{t\to s} g(\Phi_t, \Phi_s) = 0$
and thus there is some neighborhood $V$ of $s$ so that
$$
g(\Phi_t(\ker \Phi_s), \Phi_s(\ker \Phi_t)) = g(\ker \Phi_t, \ker \Phi_s) <\gamma(\Phi_t(\ker \Phi_s, N))
$$
which means that also $\Phi_s(\ker \Phi_t)$ is complemented in $X_{t,s} = X_{s,t}$ and thus the sequence
$0\to \Phi_s(\ker \Phi_t)\to X_{s,t} \to X_t\to 0$ generated by $\mathcal R_{s,t}$ splits.
\end{proof}

\begin{prop} Fix $s\in D$.
\begin{enumerate}
\item If $(\Phi_d)_{d\in D}$ is a continuous family of interpolators then there is $\varepsilon>0$ such
that if $|t-s|<\varepsilon$ and the sequence $ 0 \to \ker \Phi_s \to \mathcal H \to X_s \to 0$ splits then
each exact sequence $0\to \ker \Phi_t\to \mathcal H \to X_t \to 0$ splits and $X_t$ is isomorphic to $X_s$
\item If $(\Psi_d, \Phi_d)_{d\in D}$ is bicontinuous then there is $\varepsilon>0$ such that if
$|t-s|<\varepsilon$ and the sequence
$0 \to \ker \Psi_s \cap \ker \Phi_s \to \mathcal H \to X_{\Psi_s,\Phi_s} \to 0$ splits then each exact
sequence $ 0 \to \ker \Psi_t \cap \ker \Phi_t  \to \mathcal H \to X_{\Psi_t,\Phi_t} \to 0$ splits
and $X_{\Psi_t,\Phi_t}$ is isomorphic to $X_{\Psi_s,\Phi_s}$.
\end{enumerate}
\end{prop}
\begin{proof} Use Proposition \ref{prop:gap-stab}.\end{proof}

Examples of continuous and bicontinuous families of interpolators are provided by the general differential
methods of \cite{ckmr}. Recall the description we gave of differential method. With the same notation:

\begin{prop}
The family of couples $(\Psi_d, \Phi_d)_{d\in \mathbb A}$ is bicontinuous.
\end{prop}
\begin{proof} We use here the identification of $b=(b_n)$ with the function on $\mathbb A$ given by
$f_b(z) = \sum z^nb_n$. To prove the first part, pick $f\in \ker \Phi_s$ and define the function
$g(z) = f(z)/(z-s)$ when $z\neq s$ and $g(s)=f'(s)$.
According to \cite[Lemma 3.11]{ckmr} the function $g$, identified with its Laurent expansion
$g(z) = \sum z^n g_n$ is also al element of $\mathscr J(X, \overline B)$ and with a bound
$\|g\|\leq C\|f\|$ for a constant $C>0$ independent on $f$.
 Pick the function $(z-t)g(z)\in \ker \Phi_t$ to obtain
$$
\|f(z) - (z-t)g(z)\| = \|(z-s)g(z) - (z-t)g(z)\| = |s-t|\|g\| \leq |s-t|C\|f\|
$$
Thus $g(\ker \Phi_t, \ker \Phi_s) \leq C|t-s|$, which shows that the family $(\Phi_t)$ is continuous.

For the second part, if $f\in\ker\Psi_s\cap \ker \Phi_s$ repeat the previous argument to get the function $h\in \mathscr J(X, \overline B)$ given by $h(z) = f(z)/(z-s)$ when $z\neq s$ and
$h(s)=f'(s)$ with a bound $\|h\|\leq C\|f\|$; and then the function $g\in \mathscr J(X, \overline B)$
given by $g(z) = h(z)/(z-s)$ when $z\neq s$ and $g(s)=h'(s)$ with a bound $\|g\|\leq C\|h\|$.
Form the function $(z-t)^2g(z)\in \ker \Phi_t\cap \ker \Phi_s$ to obtain
\begin{eqnarray*}
\|f(z) - (z-t)^2g(z)\| &=& \|(z-s)^2g(z) - (z-t)^2g(z)\|\\
&=& |(z-s)^2 -(z-t)^2|\|g\|\\ &\leq& (|s|^2-|t|^2 + 2z|t-s|) C^2\|f\|
\end{eqnarray*}
Thus $\lim_{t\to s} g(\ker \Psi_t\cap \ker \Phi_t, \ker \Psi_s\cap \ker \Phi_s) =0$.
\end{proof}
The bicontinuity of couples $(\delta_t', \delta_t)_t$ associated to the complex method in the unit strip was studied in \cite{cfg}, obtaining the estimate
$$g(\cap_{0\leq j\leq n-1} \ker \delta_t^j, \cap_{0\leq j\leq n-1}\ker\delta_s^j) \leq 2(n+1) {\sf h}(t,s)$$
where ${\sf h(\cdot)}$ be the hyperbolic distance on the strip. One however has:

\begin{lemma} $g(\delta_t, \delta_t') \in \{0,1\} $
\end{lemma}
\begin{proof} The result follows from Lemma \ref{equal}: $\delta_t'(\ker \delta_t) = X_t$ by standard compatibility, while $\delta_t(\ker \delta_t')=X_t$ only when $\mathcal H = \ker \delta_t + \ker \delta_t'$ according to Theorem \ref{main}. If this is the case (the induced sequence
splits) then $g(\delta_t, \delta_t')=0$. Otherwise, $\delta_t(\ker \delta_t') $ is a proper subspace of $\delta_t'(\ker \delta_t)$.\end{proof}

\subsection{Singularity}
Recall that an exact sequence is called singular \cite{cfg} when the quotient map is a strictly singular operator.
In the particular case of exact sequences induced by a derivation $\Omega$ we will say that $\Omega$ is singular.
The paramount examples of singular exact sequences are the Kalton-Peck sequences (see Example \ref{ex:1}). The papers \cite{cfg,correa}
contain a number of criteria to determine when a derivation $\Omega$ is singular. The core in all cases is an estimate
\cite[Lemma 4.8]{cfg} (see also \cite[Lemma 2.11]{correa}) of the type

\begin{equation}\label{eme}
\rm{Average}_{\pm}\left\|\Omega_\theta\big(\sum_{i=1}^n \pm x_i\big)-\sum_{i=1}^n \Omega_\theta(\pm x_i)-
\log\frac{\pi_{X_0}(n)}{\pi_{X_1}(n)} \Big(\sum_{i=1}^n \pm x_i\Big)\right\|\leq C(\theta)
\pi_{X_0}(n)^{1-\theta} \pi_{X_1}(n)^\theta
\end{equation}

\noindent where $\pi: \rm{Ban}\to \R^\N$ is a parameter defined on a certain subclass of Banach spaces: K\"othe spaces, spaces with unconditional basis, etc. From this it follows that
if $W\subset X_\theta$ is a subspace on which the restriction $\Omega_{|W}$ is trivial then one has
\begin{equation}\label{est}\left| \log \frac{\pi_{X_0}(n)}{\pi_{X_1}(n)}\right| \pi_W (n)\leq C \pi_{X_0}(n)^{1-\theta}\pi_{X_1}(n)^{\theta}\end{equation}
from where different hypotheses on the behaviour of $\pi_{X_0}, \pi_{X_1}$ and $\pi_W$ yield different criteria to make $\Omega_\theta$ singular.

Consider now the abstract interpolation method of Cwikel, Kalton, Milman, Rochberg \cite{ckmr} briefly described as Example 1. As it is observed in \cite{ckmr}, the identification of the space $\mathscr J(\textbf X, \overline B)$ with certain analytic functions $f:\mathbb A\to X_0 + X_1$
implies that the two associated interpolators $(\Psi_s, \Phi_s)$ adopt the form $\Phi_s(f)= f(s)$ and $\Psi_s(f)= f'(s)$; i.\ e., the standard couple $(\delta_s', \delta_s)$ of the classical complex method (although in a different context). But this means that the estimates
\cite[Lemma 4.8]{cfg} and \cite[Lemma 2.11]{correa} can be reproduced almost verbatim and thus the estimates \ref{est} also holds. We present an omnibus result of this kind with three parts, each depending on a different choice of the parameter $\pi(\cdot)$ adapted to work in different situations:

\begin{itemize}
\item In the context of K\"othe spaces we consider the parameter \cite{cfg}
$$\pi^M_{X}(n) =\sup \{\|x_1+\ldots+x_n\|: \; \; \|x_j\|\leq 1\},$$ where the supremum is taken on disjointly finitely
supported families;
\item In the context of Banach spaces with Schauder basis we consider \cite{cfg} the parameter
$$\pi^A_{X}(n) =  \sup \{\|x_1+\ldots+x_n\|: \; \; \|x_j\|\leq 1;\; n<x_1<x_2< \dots < x_n\}.$$
\item In arbitrary Banach spaces we consider \cite{correa} the finite-type-like constants
$$\pi^\beta_X(n) = \inf \{ \sigma>0: \rm{Average}_\pm \left\|\sum_{j=1}^n \pm x_j\right\|\leq \sigma,\quad \|x_j\|\leq 1 \}$$
\end{itemize}

For a Banach space $X$ we denote $p_X = \sup\{p : X \text{ has type } p\}$.

\begin{prop}[\cite{cfg, correa}]\label{singu} Let $(X_0, X_1)$ be an interpolation couple and let $0<\theta<1$.
\begin{enumerate}
\item Suppose $X_0$ and $X_1$ are K\"othe function spaces so that $\lim \sup \pi^M_{X_1}(n)/ \pi^M_{X_0}(n) = \infty$, $X_\theta$ is reflexive and such that $\pi^M_W \sim \pi^M_{X_\theta}$ for every
infinite-dimensional subspace  $W\subset X_\theta$ generated by a disjoint sequence, and
$\pi^M_{X_{\theta}} \sim {\pi^M_{X_0}}^{1-\theta} {\pi^M_{X_1}}^\theta$. Then $\Omega_\theta$ is disjointly singular. In particular, on Banach spaces with unconditional basis $\Omega_\theta$ is singular.
\item Suppose $X_0$ and $X_1$ are Banach spaces with a common $1$-monotone Schauder basis. Let $1 \leq p_0 \neq p_1 \leq +\infty$, $0<\theta<1$, and
$\frac{1}{p}=\frac{1-\theta}{p_0}+\frac{\theta}{p_1}$ and assume that the spaces
$X_j$, $j=0,1$ satisfy an asymptotic upper $\ell_{p_j}$-estimate;
and that for every block-subspace $W$ of $X_\theta$, there exists a constant $C$
and for each $n$, a $C$-unconditional finite block-sequence $n < y_1 < \ldots < y_n$
in $B_W$ such that $\|y_1 + \cdots + y_n\| \geq C^{-1} n^{1/p}$
and $[y_1, \cdots, y_n]$ is  $C$-complemented in $X_\theta$. Then $\Omega_\theta$ is singular.
\item  If $(X_0, X_1)$ are arbitrary Banach spaces, assume that $X_0$ has type $p_{X_0}$ and $X_1$ has type $p_{X_1}$ and that for every infinite dimensional closed subspace $W \subset X_\theta$ has (optimal) type $p$ with $p^{-1} = (1-\theta)p_{X_0}^{-1} + \theta p_{X_1}^{-1}$. Then
    $\Omega_\theta$ is singular.\end{enumerate}
\end{prop}

We present in the same context a result in which instead of singularity of the derivation the focus is placed in the incomparability of the spaces in the scale.

\begin{prop}\label{inco}
Let $(X_0, X_1)$ be an interpolation couple of Banach spaces and let $0<\theta<1$.
\begin{enumerate}
\item Assume that $X_0$ and $X_1$ have a common unconditional basis, $\lim \sup \pi^M_{X_1}(n)/ \pi^M_{X_0}(n) = \infty$ and that for
every subspace $W\subset X_\theta$ spanned by a sequence of disjointly supported vectors $\pi^M_W(n) \sim {\pi^M_{X_0}(n)}^{1-\theta} {\pi^M_{X_1}(n)}^\theta$.
Then the spaces $X_\theta$ and $X_t$ are totally incomparable for $t<\theta$.
\item Assume that $X_0$ and $X_1$ have a common Schauder basis so that $\lim \sup \pi^A_{X_1}(n)/ \pi^A_{X_0}(n) = \infty$ and
for every block subspace $W$ of $X_\theta$ one has $\pi^A_W(n) \sim {\pi^A_{X_0}(n)}^{1-\theta} {\pi^A_{X_1}(n)}^\theta$.
Then the spaces $X_\theta$ and $X_t$ are totally incomparable for $t<\theta$.
\item Assume that $X_0$ has type $p_{X_0}$, $X_1$ has type $p_{X_1}$, and that for every closed infinite
dimensional subspace $W$ of $X_{\theta}$ the relation
$\frac{1}{p_W} = \frac{1 - \theta}{p_{X_0}} + \frac{\theta}{p_{X_1}}$ holds.
Then the spaces $X_{\theta}$ and $X_t$ are totally incomparable for $t > \theta$.
\end{enumerate}
\end{prop}
\begin{proof}
We prove (1), the other items being proved similarly: observe that if it $X_{\theta}$ and $X_t$ are not totally incomparable we may find a
normalized block-sequence $(w_i)_i$ in $X_\theta$ which is $c$-equivalent to a block-sequence $(v_i)_i$ in $X_t$, for some $c$.
For each $n \in \N$ there is a normalized disjoint sequence $(y_i)_{i=1,\ldots,n}$ of $(w_i)_i$ and $C > 0$ such that
$\|y_1+\cdots+y_n\| \geq C^{-1}\pi^M_{X_\theta}(n)$.
The sequence $(y_i)$ is $c$-equivalent to a semi-normalized disjoint sequence $(z_i)$ in $X_t$.
By assumption
$\|z_1+\ldots+z_n\| \leq c{\pi^M_{X_0}(n)}^{1-t}{\pi^M_{X_1}(n)}^t$ (we may suppose the $c$ here is the same as before).
Therefore $C^{-1}{pi^M_{X_0}(n)}^{1-\theta}{\pi^M_{X_1}(n)}^\theta \leq Cc {\pi^M_{X_0}(n)}^{1-t}{\pi^M_{X_1}(n)}^t$ and
$$\left (\frac {\pi^M_{X_1}(n)} {\pi^M_{X_0}(n)} \right)^{\theta-t} \leq cC^2,$$
which yields a contradiction.
\end{proof}

The combination between the general results in Proposition \ref{singu} and their counterparts in Proposition \ref{inco}
suggests the following problem:

\begin{prob}\label{prob:solved}
Assume  that $X_\theta$ is totally incomparable with $X_t$ for $t \neq \theta$ in a neighborhood of $\theta$.
Is $\Omega_\theta$ singular?
\end{prob}

The negative solution to this problem is delayed to \cite{ccc} since a long and contorted digression on admissible Kalton spaces of analytic functions is required.

\end{document}